\documentclass{amsart}
\usepackage{amsmath,amssymb,amsfonts,amsthm}
\usepackage{amssymb}
\usepackage{pgf}
\usepackage{todonotes}
\usepackage{epsfig}
\usepackage{mathrsfs}
\usepackage{verbatim}
\newtheorem{theorem}{Theorem}[section]

\newtheorem{proposition}[theorem]{Proposition}

\theoremstyle{definition}

\theoremstyle{remark}

\numberwithin{equation}{section}

\newcommand{\PI}{\mathrm{P}_\mathrm{I}}
\renewcommand{\PII}{\mathrm{P}_\mathrm{II}}

\renewcommand{\span}{\mathrm{span}}

\newcommand{\conv}{\mathrm{conv}}

\newcommand{\Q}{\mathbb{Q}}
\newcommand{\X}{\mathrm{X}}

\newcommand{\B}{\mathbf{B}}

\newcommand{\cf}{\mathrm{cf}}

\begin{document}

\title[Open-point games and productivity of dense-separability]{Open-point topological games and productivity of dense-separable property}

\author{Jarno Talponen}
\address{University of Eastern Finland, Institute of Mathematics, Box 111, FI-80101 Joensuu, Finland}
\email{talponen@iki.fi}

\date{\today}

\begin{abstract}
In this note we study the open-point topological games in order to analyze the least upper bound for density of dense subsets of a topological space. This way we may also 
analyze the behavior of such cardinal invariants in taking products of spaces. Various related cardinal equalities and inequalities are given. 
As an application we take a look at Banach spaces with the property (CSP) which can be formulated by stating 
that each weak-star dense linear subspace of the dual is weak-star separable.  
\end{abstract}

\maketitle

\section{Introduction}

Following \cite{LM} we call a topological space dense-separable if each of its dense subsets is separable (cf. \cite{BJ2}). 

It is known that (see \cite{JS}) that for compact spaces the dense-separability is equivalent to the countable $\pi$-weight. 
It follows that the property of being compact dense-separable space is preserved in countable products.
These considerations also involve the $\mathrm{S}$ and $\mathrm{L}$ spaces (see \cite{BJ}, \cite[Ch.11]{KV}). 

However, the problem whether a finite product of dense-separable topological spaces must be dense-separable appears not have been settled. This problem is the starting point and the main motivation in this paper.
The above problem is also closely related to some questions involving the structure of Banach spaces 
(see the discussion at the end, cf. \cite{C, Pol, talponen}).

The preservation of a given property of topological spaces in finite products is a notoriously delicate problem, see \cite{ss}.
In fact, often it is the case that a product of two spaces may fail to have that property even if both the spaces have the property, but the property of the product 
can be recovered by requiring one of the spaces to satisfy a stronger form of the given property. 

The main point of this paper is that a kind of stronger form of dense-separability which can be formulated by means of a open-point game 
is productive in a way, provided that the product enjoys a suitable tightness condition. The role of this tightness condition is to enable a `product strategy', a kind of Cartesian product of winning strategies involving the coordinate spaces.

Next some ideas behind the approach are very heuristically described.
Given a metric space $(X,d)$ there is an easy, natural and sparse way of recursively selecting a dense set, 
e.g. towards verifying the hereditary separability of the space. This is a kind of greedy algorithm.

First pick any point $x_0 \in X$. In the later stages $\beta>0$ one defines 
\[d_\beta := \sup \left\{0<r<1 \colon \exists x \in X\  B(x,r) \subset X \setminus  \bigcup_{\alpha<\beta} x_\alpha \right\}\]
and then chooses $x_\beta \in X$ such that $B(x_\beta ,d_\beta /2 ) \subset X \setminus  \bigcup_{\alpha<\beta} x_\alpha$. If the density of $X$ is $\kappa$ then 
the above construction results in a dense transfinite sequence $\{x_\beta \}_{\beta<\theta} \subset X$ with $\kappa\leq \theta < \kappa^+$.
Thus, in view of selecting a sparse dense transfinite sequence, the metric provides a quantitative description of `the order of importance' of different parts of the space in the process.

It appears reasonable trying to generalize this kind of description of `strategic importance' of different parts of the space to general topological setting. However, this case of course differs considerably from the metrizable 
case. One drastic difference is that (e.g. in a separable but non hereditarily separable space) there are some `key nodes' which are especially effective in the formation of dense subsets of the space.
Therefore topological spaces in general are not homogeneous in such a way that a simple description of the strategically important parts of the space, such as above, should appear plausible.  

The author first attempted a kind of grading of the topology and modeling a kind of monotone choice function involving the asymptotic order of importance, in a similar spirit as above. Formalizing these ideas turned out to be difficult. 
In fact, this is how the author discovered that the winning strategy of Player I in a open-point game yields the sought after asymptotic order of strategic importance of parts of the topological space.
The winning strategy also adapts to the points selected by Player II during the course of the game and this in part deals with the above inhomogeneity issue.

Thus the open-point game is like tailor-made for analyzing the subtle density-related properties of topological spaces. See \cite{aurichi_et, bella_et, DG, scheepers, zapletal} for some recent related work involving such games and cardinal invariants.

\subsection{Preliminaries and notations}
We refer to the monographs in the references for a suitable background information.
In what follows $\kappa$ stands for an infinite cardinal.
Recall that a $\pi$-basis of a topological space $X$ is a system of non-empty open sets $\{U_\gamma \colon \gamma \in \Gamma\}$ such that any non-empty open set $U \subset X$ contains some $U_\gamma$.
We denote by $\pi(X)$ the $\pi$-weight  of $X$, that is, the least cardinality of a $\pi$-basis of $X$. Similarly, the weight of a space is denoted by $w(X)$.
The density $\mathrm{d}(X)$ of a topological space $X$ is the least cardinality of a dense subset.
The tightness $t(x, X)$ of $X$ at a point $x\in X$ is the smallest cardinal $\kappa$ such that, whenever $x\in \overline{Y}$ for some subset $Y \subset X$, 
there exists a subset $Z \subset Y$, with $|Z | \leq \kappa$, such that $x\in \overline{Z}$. The tightness of $\X$ is $t(X)=\bigvee_{x\in X} t(x,X)$. 
We denote by 
\[\delta(T) = \bigvee \{\mathrm{d}(A)\colon A \subset T\ \mathrm{dense\ subset}\}\]
following e.g. \cite{BJ2}. For example, if $\delta(T) = \aleph_0$ then $T$ is dense-separable.
The closed unit ball of a Banach space $X$ is denoted by $\B_X$. Given a product 
$\prod_\gamma X_\gamma$ of topological spaces we denote the natural projection to 
a given coordinate space by $\pi_{X_\gamma}$. Unless otherwise sated, we will consider Cartesian products 
of topological spaces with the product topology.

\subsection{The game $G_{T,\theta}$}
Let $T$ be a topological space and let $\theta$ be an infinite ordinal. We will study a $2$-player perfect information game as in \cite{BJ} which we next recall. 
The game advances in stages indexed by ordinals. At each stage ($\alpha$ ordinal) Player I ($\PI$) chooses an open set $\emptyset \neq U_\alpha \subset T$ and then Player II ($\PII$) chooses a point $x_\alpha \in U_\alpha$. The game ends at the stage $\beta$ with $\overline{\bigcup_{\alpha<\beta} x_\alpha} =T$. In such a case we say that the game played has order type or length $\beta$.
If the game ends before the stage $\theta$, i.e. $\beta<\theta$ above, then $\mathrm{P}_\mathrm{I}$ wins, otherwise $\mathrm{P}_\mathrm{II}$ wins.
Thus the chances of $\mathrm{P}_\mathrm{I}$ to succeed improve as $\theta$ increases. This is a game of perfect information, so both the players know the ordinal 
representing the stage of the game as well as the whole history of the innings played in the game. 

In this game it will be beneficial for $\PI$ to always choose $U_\beta \subset T \setminus \overline{\bigcup_{\alpha<\beta} x_\alpha}$, so 
we could instead impose this alternative rule without changing any outcomes of the games. By a strategy of $\PI$ we mean a mapping
$s\colon  T^{<\theta} \to \tau\setminus \{\emptyset\}$ with the interpretation that once a transfinite sequence of points have been selected by $\PII$
the complement of the closure of the sequence is an open set containing the relevant information about the status of the game and then $s$ describes the next inning of $\PI$. The strategy of $\PII$ is a mapping $s\colon T^{<\theta} \times (\tau\setminus \{\emptyset\}) \to T$. Strictly speaking, the strategies involve only the sequences of innings played according to the rules.

Note that since $\PI$ can play redundant innings, if she has winning strategy $s$ in game $G_{T,\theta}$, then she has a winning strategy also in a version of the same game where $\PII$ must pick at least one point 
at a time (possibly uncountably many). 

Note that since 
\[\overline{\{x_\alpha \}}_{\alpha<\theta} = \overline{\overline{\{x_\alpha \}}_{\alpha<\beta} \cup \{x_\alpha \}_{\beta\leq \alpha<\theta}},\quad \beta<\theta\]
the strategies of both players could be alternatively formulated in terms of closures of the sequence of points selected, provided that information on the stage of the game is incorporated in the strategy,
cf. \cite{berner}.

We denote by $\mathrm{gd}(T)$ the cardinal invariant associated to space $T$ as follows: Let $\mathrm{gd}(T)$ be the least cardinal $\kappa$ such that 
$\mathrm{P}_\mathrm{I}$ has a winning strategy in the  game $G_{T,\kappa^+}$ (clearly $\mathrm{gd}(T) \leq |T|$).

\section{Results}

For the sake of convenience we will begin with some basic observations, some of which are likely known to the specialists in the topic, cf. \cite{berner,BJ, juhasz}.

\subsection{Basic observations}

\begin{proposition}\label{thm: ineq}
For any topological space $T$ we have
\[\delta(T) \leq \mathrm{gd}(T) \leq \pi(T).\]
In particular, if $T$ has a countable $\pi$-basis, then $\mathrm{P}_\mathrm{I}$ has a winning strategy in  $G_{T,\omega_1 }$, and this condition in turn 
implies that $T$ is dense-separable.  Moreover, if $\delta(T) =\pi(T)$ then all the games $G_{T,\kappa}$ are determined. 
\end{proposition}
\begin{proof}
To check the left hand inequality, assume that there is a dense set $A \subset T$ with $\mathrm{d}(A)=\kappa$.  Then $\mathrm{P}_\mathrm{II}$ may use the following strategy:
Namely, always picking an element $x_\alpha \in U_\alpha \cap A$. The game does not terminate before the stage $\kappa$ because of the density 
assumption. Therefore $\mathrm{P}_\mathrm{I}$ does not have a strategy securing a $G_{T,\kappa}$ win for any $\kappa < \delta(T)$. 

To check the right hand inequality, observe that $\mathrm{P}_\mathrm{I}$ has a winning strategy in a game $G_{T, \pi(T)^+ }$ as follows. Let us first well-order
a $\pi$-basis of $T$ in the form $\{U_{\alpha} \}_{\alpha < \pi(T)}$. A suitable winning strategy for $\mathrm{P}_\mathrm{I}$ is choosing 
\[V_\alpha = U_{\beta_\alpha}\]
where $\beta_\alpha$ are defined by means of transfinite recursion obeying the rule
\[\beta_{\alpha} = \bigwedge \left\{\nu < \pi(T) \colon U_\nu \cap \overline{\bigcup_{\theta< \alpha} x_\theta } = \emptyset\right\}.\]
The game will terminate at latest at stage $\pi(T)$.

In the case with $\delta(T) =\pi(T)$ the game $G_{T,\kappa}$ is determined, since according to the above observations $\PII$ wins if $\kappa \leq \delta(T)$ and $\PI$ wins if $\kappa> \pi(T)$. 
\end{proof}

The following $2$ observations seem useful facts although they are easy exercises. 

\begin{proposition}
Suppose that $S \subset T$ satisfies $\overline{\mathrm{int}(\overline{S})}=\overline{S}$
(e.g. if $S$ is open or dense). Then 
\[\mathrm{gd}(S) \leq \mathrm{gd}(T) .\]
\end{proposition}
\qed

\begin{proposition}
Let $X$ be a topological space. Suppose that 
$\{x_\alpha\}_{\alpha<\kappa} \subset X$ is a dense family of points with $t(x_\alpha, X) \leq \kappa$. Then 
$\delta (X) \leq \kappa$. 
\end{proposition}
\qed

Next it turns out that if $\PI$ can control the game well then the cardinal invariants appearing in the inequalities in \ref{thm: ineq} coincide.

\begin{theorem}\label{thm: card_inv}
Let $X$ be a topological space and $\theta$ be the order type (i.e. length) of any game played in $G_{X, |X|^+ }$ won by $\PI$.
Then $\cf(\theta) \leq \delta(X)$.

Moreover, if $\PI$ has a strategy which makes the game end exactly at the stage $\kappa$, then 
$\kappa=\mathrm{gd}(X)=\delta(X)=\mathrm{d}(X)$.
\end{theorem}
\begin{proof}
To verify the first part, let $\{x_\alpha \}_{\alpha< \theta} \subset X$ be the family resulting from the choices of $\mathrm{P}_\mathrm{II}$ in a given game played. This is a dense subset in $X$.
Then it is easy to see that 
\[\cf (\theta) \leq \mathrm{d}(\{x_\alpha \}_{\alpha< \theta}) \leq \delta(X)\] 
since according to the choice of the family there are no dense subsets $\Lambda \subset \{x_\alpha \}_{\alpha< \theta}$ with $|\Lambda| < \cf(\theta)$.

To check the second part of the statement, fix $\kappa$ to be the length of the game which can be enforced by $\PI$, regardless of the strategy played by $\PII$. 
We observe that by the argument of Proposition \ref{thm: ineq} we must have $\kappa=\delta(X)=\mathrm{d}(X)$. Indeed, $\PII$ could play according to any 
dense subset, so that all the dense subsets of $X$ must have the same density.  
Clearly $\kappa \geq \mathrm{gd}(X)$ by the definition of the invariant. Since $\delta(X)\leq \mathrm{gd}(X)$ we obtain $\kappa=\mathrm{gd}(X)=\delta(X)=\mathrm{d}(X)$.
\end{proof}

It is well-known that if a topological space $X$ is metrizable then the cardinal invariants in 
\ref{thm: card_inv} coincide. The following result gives a better criterion which ensures 
the same conclusion.

\begin{theorem}\label{thm: metrizable}
Suppose that $X$ embeds in a product of $\mathrm{d}(X)$-many metrizable spaces, or, equivalently, $X$ is a completely regular space uniformizable by $\mathrm{d}(X)$-many
pseudometrics. Then
\[\mathrm{d}(X)=\delta(X) = \mathrm{gd}(X)=\pi(X)=w(X) .\] 
\end{theorem}
\begin{proof}
Since the above equalities hold generally as inequalities, it suffices to show that $w(X)\leq \mathrm{d}(X)$.
Let $X\subset Z:=\prod_{\alpha< \mathrm{d}(X)} X_\alpha$ be a product of metric spaces $X_\alpha$ such that $X$ embeds in this product. 
We denote the canonical projections by $\pi_\alpha \colon Z \to X_\alpha$.
We may assume without loss of generality that $\mathrm{d}(X_\alpha) \leq \mathrm{d}(X)$ for each $\alpha$.
Thus $w(X_\alpha )=\mathrm{d}(X_\alpha)\leq \mathrm{d}(X)$ for each $\alpha$. 

Let $\{U_{\alpha, \beta}\}_{\beta<\mathrm{d}(X)}$ be a base for each $X_\alpha$. Let $\mathcal{F}$ be the collection of mappings of the form 
$f \colon A \to \mathrm{d}(X) \times \bigcup_{\alpha,\beta}  \{U^{(\alpha)}_{\beta}\}_{\beta<\mathrm{d}(X)}$, 
$A\in \mathrm{d}(X)^{<\omega}$, with $f\colon y \mapsto (\alpha, U^{(\alpha)}_\beta )$ where $\pi_1 f$ (the projection of $f$ to the first coordinate) is an injection.

Then the sets of the form 
\[\prod_{\substack{\alpha =\pi_1 f a\\ a\in A}} \pi_2 f(a) \quad \times \prod_{\beta \in \mathrm{d}(X)\setminus \pi_1 fA} X_\beta,
\quad f\in \mathcal{F},\ \mathrm{dom}(f)=A\in \mathrm{d}(X)^{<\omega} \]
form a base for the product topology of $Z$. The cardinality of the base is 
\[|\mathrm{d}(X)^{<\omega}\times  (\mathrm{d}(X)\times  \mathrm{d}(X)\times  \mathrm{d}(X))^{<\omega}|=\mathrm{d}(X) .\] 
Thus $X$ inherits a base of cardinality $\leq\mathrm{d}(X)$.
\end{proof}

\subsection{Main results}

The proof of the following result is a baby version of the argument of Theorem \ref{thm: main}.

\begin{theorem}
For any topological spaces $X$ and $Y$ we have 
\[\mathrm{gd}(X\times Y) \leq \pi(X) \vee \mathrm{gd}(Y).\]
\end{theorem}
\begin{proof}
Let $\kappa=\pi(X) \vee \mathrm{gd}(Y)$. Let $s_Y$ be the winning strategy of $\PI$ in $G_{Y,\kappa^+ }$.
Let $\{U_\alpha \}_{\alpha<\kappa}$ be a $\pi$-base of $X$.
The proof uses the facts that $|\kappa \times \kappa^+|=\kappa^+$ and the union of $\kappa$-many non-cofinal sequences of $\kappa^+$ 
is not cofinal in $\kappa^+$. 

The strategy of the argument is as follows. Player I may essentially run $\pi(X)$-many parallel games 
choosing open sets $U_\alpha \times V$. Here $U_\alpha \subset X$, $\alpha<\kappa$, are open sets appearing in a suitable $\pi$-base, corresponding to the $\alpha$:th parallel game, 
and $V$ is a $\mathrm{P}_\mathrm{I}$-inning in $G_{Y, \kappa^+ }$.
More precisely, the game advances as follows at $\alpha$:th parallel game. 
First $\mathrm{P}_\mathrm{I}$ may choose $W_0 = U_\alpha \times Y$. 
Then $\mathrm{P}_\mathrm{II}$ chooses a point in $z_0 \in W_0 = U_\alpha \times Y$. 
At later stages $\sigma < \kappa^+$ (whenever played) $\PI$ `projects' the game to $Y$ by applying her strategy to choose an open subset $V \subset Y\setminus \overline{\pi_Y (z_\beta \colon \beta<\sigma)}$.

Player I may aggregate the parallel games as a single game by using the ordinal arithmetic to express the stages as follows
\[\gamma = \kappa \cdot \beta + \alpha\]
where $\alpha<\kappa$ is the number of the parallel game played, $\beta<\kappa^+$ is the stage of $\alpha$:th game and $\gamma$ is the stage 
in the aggregate game involving $X \times Y$. Each of the $\kappa$-many parallel subgames terminates before stage $\kappa^+$, thus there is 
$\gamma < \kappa^+$ such that there are no innings to be played in $G_{X\times Y, \kappa^+ }$ after stage $\gamma$.

To fill in the details, a winning strategy $s$ of $\PI$ in $G_{X\times Y, \kappa^+}$ is can be defined as follows:
\begin{enumerate}
\item $\PI$-innings for $\gamma =\kappa\cdot \beta$, $\beta<\kappa^+$, (zero-stage in the subgames) are $X\times Y$. 
\item $\PI$-innings for $\gamma =\kappa \cdot \beta + \alpha$ with $0<\beta<\kappa^+$, $0<\alpha<\kappa$:
\[U_\alpha \times s_Y (\{\pi_Y (x_{\kappa \cdot \sigma + \alpha})\}_{\sigma<\beta})
\ \text{if}\ \overline{\{\pi_Y (x_{\kappa \cdot \sigma + \alpha})\}}_{\sigma<\beta}\neq Y \]
\[\text{and}\ U_\alpha \times Y\ \text{otherwise}.\]
\end{enumerate}
For technical reasons we let the game(s) run here possibly a bit longer even after the points picked by $\PII$ 
are already dense.
For each $\alpha<\kappa$ let 
\[\beta_\alpha := \bigwedge \{\beta<\kappa^+ \colon \overline{\{\pi_Y (x_{\kappa \cdot \sigma + \alpha})\}}_{\sigma<\beta} = Y\}.\]
By the definition of $s_Y$ we obtain $\beta_\alpha<\kappa^+$ for all $\alpha<\kappa$. Put 
\[\hat{\beta}:=\bigvee_{\alpha<\kappa}  \kappa \cdot \beta_\alpha + \alpha .\]
Observe that $\hat{\beta}<\kappa^+$. By the definitions of the product topology and $\pi$-basis, 
and the choices of $\beta_\alpha$ we see that $\{x_{\theta}\}_{\theta<\hat{\beta}}\subset X \times Y$
is dense. This means that $s$ is a winning strategy for $\PI$ in $G_{X\times Y, \kappa^+ }$.
\end{proof}

We will consider a kind of fan tightness condition involving products of spaces.
Let $\kappa$ be an infinite cardinal and $\{X_\alpha \}_{\alpha <\sigma}$, $\sigma\leq \kappa$, be a family of topological spaces. Suppose that for each finite subset $\Gamma \subset \sigma$, each $\alpha<\sigma$ 
and each non-empty open subset $U \subset \prod_{\gamma\in \Gamma} X_\gamma$ 
there is a family of open subsets $\{V_\beta \}_{\beta<\kappa} \subset \prod_{\gamma\in \Gamma} X_\gamma$ 
such that whenever $A\subset \prod_{\gamma\in \Gamma} X_\gamma$ is such that for all 
$\beta < \kappa$ and $\gamma \in \Gamma$ 
\begin{equation}\label{eq: char0}
\mathrm{A}_\beta := A \cap V_\beta \ \text{satisfies}\  \overline{\pi_{X_\gamma} \left(\mathrm{A}_\beta \right)} = \overline{\pi_{X_\gamma} \left(V_\beta \right)},
\end{equation}
then there is $\gamma \in \Gamma$ such that 
\begin{equation}\label{eq: char}
\pi_{X_\gamma} (U \setminus \overline{\mathrm{A}} ) \subsetneq \pi_{X_\gamma} (U). 
\end{equation}

It is easy to see that this is for instance the case if for finitely many $\alpha<\sigma$ it holds that 
each non-empty open subset $V \subset X_\alpha$ contains a non-empty open subset $W \subset V$ with $\pi(W)\leq \kappa$ and $\pi(X_\beta) \leq \kappa$ for each $\beta<\sigma$ different from the above $\alpha$:s.

The next result has some bearing on the question in the introduction regarding the preservation of dense-separable spaces in finite or countable products.

\begin{theorem}\label{thm: main}
Let $\sigma\leq \kappa$ and $\{X_\alpha \}_{\alpha <\sigma}$ be a family of spaces with 
$\mathrm{gd}(X_\alpha  ) \leq \kappa$ and suppose that the above tightness condition involving 
\eqref{eq: char0}-\eqref{eq: char} holds. Then $\mathrm{gd}(\prod_{\alpha<\sigma} X_\alpha ) \leq \kappa$.
\end{theorem}
\begin{proof}[Sketch of Proof] 
Without loss of generality we may concentrate on the case $\sigma = \kappa$.
The innings of $\PI$ resulting in the termination of the game before the stage $\kappa^+$ are described in terms of a decomposition by means of suitable subgames.
    
Let $\{\Gamma_\beta \}_{\alpha<\kappa}$ be the indexing of $\kappa^{<\omega}\setminus \{\emptyset\}$. 
Let $s_\alpha$, $0\leq \alpha<\kappa$, be some $\PI$ $G_{X_\alpha , \kappa^+ }$ winning strategies.  
During the course of the game we will construct by transfinite recursion an auxiliary partial mapping 
$g \colon \kappa^+ \to \kappa\times \kappa^+ \times \kappa\times \kappa^+$ which will be increasing 
when we consider $\kappa\times \kappa^+ \times \kappa \times \kappa^+$ in the lexicographic order. We denote by $\pi_i$ the projection to $i$:th coordinate, $i=1,2,3,4$. 

Let the first inning of $\PI$ be $U_0 = \prod_{\alpha<\kappa} X_\alpha $. Let $x_0 \in \prod_{\alpha<\kappa} X_\alpha $ be the first inning of $\PII$.
At each stage $\nu < \kappa^+$ of the game we denote by $A_\nu$ the set of points selected by $\PII$ before stage $\nu$.

Let $g(0)=(0,0,0,0)$. 
If $g(\delta)=(\alpha,\beta,\eta, \epsilon)$ where $\epsilon<\kappa^+$, then $g(\delta + 1)= (\alpha,\beta,\gamma, \epsilon + 1)$. There are gaps on the ordering, depending on the course of the game. 
Suppose that $\pi_4 g(\delta)\nearrow \lambda$, a limit ordinal, as $\delta \nearrow \mu$,
a limit ordinal, where the first $3$ coordinates are constant on an interval with $\delta \in (\nu,\mu)$, $\nu<\mu$. Then $g(\mu) = (\alpha_0, \beta_0 , \eta_0 ,\lambda)$ (typically), or 
$g(\mu)=(\alpha_0, \beta_0 , \eta_0 +1, 0)$, depending on the course of the game. 
Similarly, the other coordinate functions $\pi_i g$, $i=1,2,3$, behave in a semi-continuous manner. 

\noindent{\bf On the interpretation of $g$.} The variable of $g$ express different stages of the game. It is useful to think of the game as 
an aggregate of suitable subgames or phases. The triple $(\alpha,\beta,\eta)$ can be thought of as a phase of the game and $\epsilon$ as a stage in this phase. Similarly, $\eta$ can be thought of as a subphase of 
$\beta$, which, in turn, is a subphase of $\alpha$. The stages $\delta$ belonging to a phase $(\alpha_0 ,\beta_0 ,\eta_0 )$ are given by
\[\{\delta < \kappa^+ \colon \exists\ \epsilon<\kappa^+ \ (\ g(\delta)=(\alpha_0 ,\beta_0 ,\eta_0 ,\epsilon)\ ) \}.\]

\noindent{\bf 1st coordinate:} $\alpha$ corresponds to  the finite index sets $\Gamma_\alpha$. The player $\PI$ drives $\PII$ choosing a subset of the product space 
which is eventually dense. According to the definition of product topology it suffices in order for $A$ to be dense in the product space that 
$\pi_{\Gamma_\alpha} A$ is dense in the corresponding finite product for all $\Gamma_\alpha$. There are $\kappa$ many of these phases, thus 
it suffices for $\PI$ in each phase $\alpha$ to yield a set $A^{(\alpha)} \subset 
\prod_{\theta<\kappa} X_\theta$, of $\PII$ innings, such that $|A^{(\alpha)}|<\kappa^+$ and $\pi_{\Gamma_\alpha}  A^{(\alpha)} \subset \prod_{\gamma\in \Gamma_\alpha} X_\gamma$ is dense.

\noindent{\bf 2nd coordinate:} For each $\Gamma_\alpha$, let $s_\gamma$ be the winning strategies concerning the respective spaces $X_\gamma$,  $\gamma \in \Gamma_\alpha$.
These strategies use the information about the stage and closure of the points selected and this must be formally handled. Let 
$\{A^{(\theta)}\}_\theta$ be an increasing family of closed sets of $\prod_{\gamma\in \Gamma_\alpha}X_\gamma$ generated by lower level 
innings of $\PII$. The input of strategy $s_\gamma$ is, roughly speaking, 
\[\overline{X_\gamma \setminus \pi_\gamma \left(\left(\prod_{\gamma\in \Gamma_\alpha}X_\gamma \right) \setminus A^{(\theta)}\right)}\]
(recall the discussion after introduction). We will effectively execute these strategies in lower levels.

There are a few technical issues at this stage that must be addressed. First is that since we are picking points in the product, instead of $X_\gamma$
itself, it could happen that the game projected in the spaces $X_\gamma$ stalls in the sense that even after innings played in the product 
the above closure remains the same. This problem will be dealt with at the next level below. It follows that some of the projected games may stall but at least one will proceed at each step. Consequently, if the $\beta$-phase subgame is played $\kappa^+$ innings, then by the pigeonhole principle there are $\kappa^+$ proper innings played in $X_{\gamma_0}$ for some $\gamma_0 \in \Gamma_\alpha$, since $\Gamma_\alpha$ is finite. To be more accurate, one must relabel the innings in the game $G_{X_{\gamma_0},\kappa^+ }$ in order to dispense with the void innings $\beta$ involving $X_\gamma$ and use the axiom of choice to pick a single point from each set of the form 
\[\overline{X_\gamma \setminus \pi_\gamma \left(\prod_{\gamma\in \Gamma_\alpha}X_\gamma \setminus A^{(\theta+1)}\right)} 
\ \setminus\  \overline{X_\gamma \setminus \pi_\gamma \left(\prod_{\gamma\in \Gamma_\alpha}X_\gamma \setminus A^{(\theta)}\right)}\]
in order to match the formal specifications of the strategy $s_{\gamma_0}$. But since $s_{\gamma_0}$ is a $\PI$ winning strategy for that game, the game ends
after $<\kappa^+$ innings (and this way we see that the relabeling does not essentially change the length of the subgame).
Thus $<\kappa^+$ innings are played at each phase $(\alpha,\beta)$.

\noindent{\bf 3rd coordinate:} $\eta<\kappa$ labels a family $V_\eta$ involving the tightness assumption 
(as in \eqref{eq: char0}-\eqref{eq: char}) where $U$ is replaced by the Cartesian product of the 
open sets given by the strategies in the previous level, 
$\prod_{\gamma \in \Gamma_\alpha}s_\gamma$. The sole purpose of this level is to keep the game moving
when projected in the coordinate spaces.

\noindent{\bf 4th coordinate:} $\epsilon$ denotes the order type of the inning played corresponding to $V_\eta$. We obtain by essentially playing $|\Gamma_\alpha |$-many subgames in a row, 
subgames projected to coordinates $\pi_{X_\gamma} V_\eta$, $\gamma \in \Gamma_\alpha$, that there is a game on the product space whose points $A=A^{(\alpha,\beta,\eta )}$ satisfy \eqref{eq: char0} for 
$\gamma \in \Gamma_\alpha$. Moreover,  by using the fact that $\mathrm{gd}(X_\gamma )\leq \kappa$, $\PI$ can choose a strategy such that the 
length of the subgame is $< \kappa^+$, regardless of the innings of $\PII$. Thus, 
$\epsilon< \mu_{\alpha,\beta,\eta}$ where the exact length of the phase $(\alpha,\beta,\eta)$, $\mu_{\alpha,\beta,\eta}$, depends on the course of the game but satisfies $\mu_{\alpha,\beta,\eta}<\kappa^+$.
\end{proof}

Under weak separation axioms it can happen that dense-separability fails but it does so only for dense subsets with high cardinality.

\begin{theorem}
Given any uncountable cardinal $\kappa$ there is a $T_1$ space $X$ such that $|X|=\mathrm{gd}(X) = \kappa$ and any dense subspace $A\subset X$ with $|A|<\kappa$ is separable.
Moreover, if $\mu$ is a measurable cardinal and there is a $\mu$-complete free ultrafilter on $\kappa$ then 
we additionally have the following property: 
Any intersection of $<\!\!\mu$-many dense non-separable subspaces of $X$ is dense and has density $\geq\mu$.
\end{theorem}
\begin{proof}
Fix an uncountable cardinal $\kappa$. As a set, we let $X$ be the disjoint union of $\kappa$ and $\omega$. 
Also, in what follows we will consider these sets disjoint. 

Let $\mathcal{F}$ be the filter on $\kappa$ consisting of sets whose complement has cardinality $<\kappa$.
Extend $\mathcal{F}$ to an ultrafilter $\mathcal{V}$ and let $\mathcal{U}$ be a free ultrafilter on $\omega$.
Define the base $\mathcal{B}$ for the topology by taking all sets of the form $V\cup U$, $V\in \mathcal{V},\ U\in \mathcal{U}$. The required topology is then generated by this base.

Now, suppose that $A \subset X$ is a dense set with $|A|<\kappa$. This means that $A\cap \kappa \notin \mathcal{V}$. Therefore 
$A\cap \omega \in \mathcal{U}$. Therefore $A\cap \omega$ intersects any $N \in \mathcal{U}$, and, consequently $A\cap \omega$  
is dense in $X$. Thus $A$ is separable.

Note that $\PII$ has a winning strategy in a game of length $\kappa$ by picking points in $\kappa$. Indeed, for any subset $I \subset \kappa$, 
$|I|<\kappa$, there is an open set of $X$ not intersecting $I$ by the choice of $\mathcal{V}$, so that $I$ is not dense. 

In the latter part of the statement we choose $\mathcal{V}$ to be $\mu$-complete. If $A_\gamma \subset X$, $\gamma<\lambda<\mu$, is a family of  
(a) non-separable and (b) dense sets, then it follows that (a) $A_\gamma \cap \omega \notin \mathcal{U}$ and (a-b) $A_\gamma \cap \kappa \in \mathcal{V}$ hold for all $\gamma$. 
Thus $(\bigcap_{\gamma<\lambda} A_\gamma) \cap \kappa=\bigcap_{\gamma<\lambda} (A_\gamma \cap \kappa) \in \mathcal{V}$, by the completeness of the filter, 
so that this intersection is dense in $X$.
\end{proof}

\subsection{Application to Banach spaces}
There is a kind of Banach space (or linear) version of the dense-separable property and this is discussed below. Next we will give some consequences of Theorem \ref{thm: metrizable} involving Banach spaces.

\begin{theorem}
Let $X$ be a Banach space, $\mathrm{d}(X)=\kappa$, and let $E \subset L(X, \ell^\infty (\mu))$, $\kappa \le \mu$,  be a closed subspace which separates 
the points of $X$. Let $\tau=\sigma(X, E)$ be the locally convex topology on $X$ induced by $E$. 
Suppose that $\mathrm{d}(E)\leq \mathrm{d}(X,\tau )$. 
Then $(X,\tau )$ satisfies
\[\mathrm{d}(\B_{X},\tau )=w(\B_{X},\tau ).\] 
\end{theorem}

Before the proof we first we note that by a standard argument $X$ can be linearly isometrically embedded in $\ell^\infty (\kappa)$. Thus we may regard $L(X) \subset L(X, \ell^\infty (\kappa)) \subset  L(X, \ell^\infty (\mu)) $ isometrically as linear subspaces in a natural way. Thus $E$ can be considered for instance any closed subspace of $L(X)$.

\begin{proof}
Let $A \subset E$ be a dense subset with $|A|=\mathrm{d}(X,\tau )$. Consider the linear mapping 
\[\mathbf{T}\colon X \to \prod_A \ell^\infty (\mu),\quad x\mapsto \{T(x)\}_{T\in A}.\]
This is injective since $E$ and thus $A$ separate the points of $X$.
It is easy to see that since $A$ is dense in $E$ the image $\mathbf{T} (\B_X ) \subset \prod_A \ell^\infty (\mu)$ 
considered in the relative topology of the product topology is in fact topologized by $\sigma(X, E)$. Finally, Theorem \ref{thm: metrizable} applies.
\end{proof}

In the same vein we observe the following known fact by thinking of the proof of the Banach-Alaoglu Theorem.

\begin{proposition}
Let $X$ be a Banach space. Then $w(\B_{X^*} ,\mathrm{w}^* )\leq  \mathrm{d}(X,\|\cdot\|)$. \qed
\end{proposition}

Recall Corson's property (C) of Banach spaces which is a convex version of weak-star countable tightness of the dual: A Banach space $\X$ has property (C) if any collection of closed convex subsets of $\X$ with empty intersection has as a countable subcollection with empty intersection, or, equivalently, if for any $A \subset \X^*$ and $a \in \overline{A}^{\mathrm{w}^*}$ there is a countable subset $A_0 \subset A$ with $a \in \overline{\conv}^{\mathrm{w}^*} (A_0 )$, see \cite{C, Pol}.

The Countable Separation Property (CSP) (see \cite{talponen}) of a Banach space $\X$ is defined as follows: If $\mathcal{F} \subset \X^*$ is any family of functionals which separates $\X$ then there is already a countable separating subfamily $\mathcal{F}_0 \subset \mathcal{F}$. An alternative way to express the same thing is that 
if $\mathcal{F} \subset \X^*$ satisfies $\overline{\span}^{\mathrm{w}*}(\mathcal{F}) = \X^*$, then there is a countable subset $\mathcal{F}_0 \subset \mathcal{F}$ such that $\overline{\span}^{\mathrm{w}*}(\mathcal{F}) = \X^*$. If $K$ is a compact regular space such that the space $C(K)$ has (CSP), then $K$ is dense-separable 
(\cite{talponen}). Moreover, if $C(K)$ has property (C) then the separability of $K$ implies the dense-separability (\cite{Pol}).

\begin{proposition}
Let $\X$ be a Banach space.
If the dual space $\X^*$ considered in the weak-star topology is dense-separable, then $\X$ has the Countable Separation Property.
 
Suppose that $\X_n$, $n<\omega$, are Banach spaces satisfying Corson's property (C) and have a weak-star separable dual. Then the infinite direct sum $\bigoplus_n \X_n$  
taken in the $c_0$-sense has the CSP.
\end{proposition}
\begin{proof}
To verify the first part of the statement, let $\mathcal{F} \subset \X^* $ be a separating family. The separating property of $\mathcal{F} $ is equivalent to stating that $\overline{\span}^{  \mathrm{w}^* }(\mathcal{F} )= \X^*$.
Thus the same conclusion holds for the $\Q$-linear span. According to the weak-star dense-separability of $\X^*$ there is a countable subset $\Lambda \subset \span_\Q (\mathcal{F} )$ 
which is weak-star dense in $\X^*$. Consequently, there is a countable $\mathcal{F}_0 \subset  \mathcal{F}$ with $\overline{\span}^{  \mathrm{w}^* }(\mathcal{F}_0 )= \X^*$, so that 
$\mathcal{F}_0$ separates $\X$. 
  
To verify the second part of the statement, we first observe that $(\bigoplus_n \X_n )^* = \bigoplus_n \X_{n}^*$, where the latter infinite direct product is taken in the $\ell^1$-sense.
Thus $(\bigoplus_n \X_n )^*$ is weak-star separable. We use the fact that property $(C)$ is stable under $c_0$-direct sum, see \cite{Pol}. Thus, $\bigoplus_n \X_n$ has weak-star separable dual and has 
property (C). It is known that property (C) and weak-star separable dual imply CSP, see \cite{talponen}, so that the claim follows. 
\end{proof}

It is not known if Banach spaces with the CSP are preserved in finite direct sums in general.

\subsection*{Acknowledgments}

This work has been financially supported  by V\"ais\"al\"a foundation's and the Finnish Cultural Foundation's research grants and Academy of Finland Project \# 268009.

\end{document}